\documentclass[12pt]{amsart}

\usepackage{graphics}
\usepackage{amsmath}
\usepackage{amsfonts}
\usepackage{amssymb}
\usepackage{amscd}
\usepackage{mathrsfs}
\usepackage{enumerate}
\usepackage{amsthm}
\input xy
\xyoption{all}

\newcommand{\CC}{{\mathbb C}}

\newcommand{\RR}{{\mathbb R}}

\newcommand{\Cont}{{\mathscr C}}
\newcommand{\NN}{{\mathbb N}}

\DeclareMathOperator{\Aut}{Aut}
\DeclareMathOperator{\Aff}{Aff}

\DeclareMathOperator{\SO}{SO}
\DeclareMathOperator{\GL}{GL}

\newtheorem{theorem}{\bf Theorem}
\newtheorem{lemma}{\bf Lemma}
\newtheorem{proposition}{\bf Proposition}
\newtheorem{corollary}{\bf Corollary}

\begin{document}

\title{Acyclic embeddings of open Riemann surfaces into new examples of elliptic manifolds}

\author{Tyson Ritter}
\address{Tyson Ritter, School of Mathematical Sciences, University of Adelaide, Adelaide SA 5005, Australia}
\email{tyson.ritter@adelaide.edu.au}


\date{1 July 2011.}


\begin{abstract}
The geometric notion of ellipticity for complex manifolds was introduced by Gromov in his seminal 1989 paper on the Oka principle, and is a sufficient condition for a manifold to be Oka. In the current paper we present contributions to three open questions involving elliptic and Oka manifolds. We show that quotients of $\CC^n$ by discrete groups of affine transformations are elliptic. Combined with an example of Margulis, this yields new examples of elliptic manifolds with free fundamental groups and vanishing higher homotopy. Finally we show that every open Riemann surface embeds acyclically into an elliptic manifold, giving a partial answer to a question of L\'arusson.
\end{abstract}

\maketitle

\section{Introduction}
\label{sec_introduction}

\noindent The field of Oka theory within complex geometry is a relatively new area of research that has undergone rapid development in recent years. With roots in the work of Oka in 1939 \cite{Oka:1939} and later extensions by Grauert \cite{Grauert:1957}, Gromov's seminal paper in 1989 \cite{Gromov:1989} set the stage for modern developments by Forstneri\v c, Prezelj, L\'arusson and others. For a detailed survey of the current state of Oka theory we refer the reader to the recent article \cite{Forstneric:2010} by Forstneri\v c and L\'arusson.

A complex manifold is said to be \emph{Oka} if it satisfies any of a number of equivalent conditions, all of which state in some way that the manifold has many holomorphic maps into it, from affine space $\CC^n$. In this sense, Oka manifolds can be thought of as being dual to Stein manifolds, which possess many maps from them into $\CC^n$. The simplest Oka condition is the \emph{convex approximation property}, which states that $M$ is Oka if every holomorphic map $K \to M$, where $K$ is a compact convex subset of $\CC^n$, can be approximated uniformly on $K$ by holomorphic maps $\CC^n \to M$. In \cite{Gromov:1989}, Gromov introduced a useful sufficient geometric condition for a manifold to be Oka that can often be verified in practice, called \emph{ellipticity}. A complex manifold $M$ is said to be \emph{elliptic} if there exists a holomorphic vector bundle $E \to M$ together with a holomorphic map $s : E \to M$ called a \emph{dominating spray}, such that $s(0_x) = x$ and $s|_{E_x}: E_x \to M$ is a submersion at $0_x$, for all $x \in M$.

Many questions exist relating to Oka and elliptic manifolds, and in this paper we make contributions to three open problems, which we now describe.

As mentioned above, Gromov showed that ellipticity is a sufficient condition for a manifold $M$ to be Oka, and if $M$ is Stein then it is also necessary \cite[3.2.A]{Gromov:1989} (see also \cite[Thm.\ 2]{Larusson:2005}). The question of whether all Oka manifolds are elliptic remains open. In Section \ref{sec_ellipticquotients} we give a sufficient condition for a quotient manifold of $\CC^n$ to be elliptic and show that quotients of $\CC^n$ by discrete groups of affine transformations satisfy this condition. Since quotients of $\CC^n$ are Oka, we give in this special case a positive answer to the question of whether Oka manifolds are elliptic. Note that it is not clear whether such quotients are Stein.

Despite the importance of elliptic manifolds, the list of known examples is relatively short \cite[Sec.\ 5]{Forstneric:2010}. In particular, it is of interest to know what possible homotopy types elliptic manifolds may have. In Section \ref{sec_newexamples} we apply the results of Section \ref{sec_ellipticquotients} to an example of Margulis to give new examples of elliptic manifolds as affine quotients of $\CC^3$. We then show that elliptic manifolds may have any free group of countable rank as fundamental group, with all higher homotopy groups vanishing. Applying a result of Baumslag and Roseblade on subgroups of the direct product of free groups, we conclude that there are continuum-many elliptic manifolds of distinct homotopy type.

In the holomorphic homotopy theory of L\'arusson \cite{Larusson:2003,Larusson:2004,Larusson:2005}, the question naturally arises whether every Stein manifold can be acyclically properly holomorphically embedded into an elliptic Stein manifold. We call a map between manifolds \emph{acyclic} if it is a homotopy equivalence. This question was answered affirmatively in \cite{Ritter:2011} for open Riemann surfaces with abelian fundamental group. In Section \ref{sec_acyclicembeddings} we use the new examples of elliptic manifolds from Section \ref{sec_newexamples} to give a partial answer to this question for one-dimensional Stein manifolds by showing that every open Riemann surface has an acyclic proper holomorphic embedding into an elliptic manifold.

I thank Finnur L\'arusson for helpful discussions during the preparation of this paper.

\section{Elliptic quotients of $\CC^n$}
\label{sec_ellipticquotients}

\noindent
As discussed in Section \ref{sec_introduction}, while it is known that elliptic manifolds are Oka, and that Stein Oka manifolds are elliptic, it remains an open question whether Oka manifolds are elliptic in general. This appears to be a difficult problem as there is no known way to construct a holomorphic vector bundle with dominating spray over an arbitrary Oka manifold.

In this section we restrict ourselves to considering quotient manifolds $M = \CC^n / \Gamma$ of Euclidean space $\CC^n$, where $\Gamma \subset \Aut(\CC^n)$ is a discrete group of holomorphic automorphisms of $\CC^n$ acting freely and properly discontinuously on $\CC^n$. The property of being Oka passes down from $\CC^n$ through the covering map \cite[Cor.\ 3.7]{Forstneric:2010}, making $M$ an Oka manifold. On the other hand, even though $\CC^n$ is elliptic, no general method is known for pushing ellipticity down to $M$ via the covering map. However, in the special case when $\Gamma$ is a group of affine automorphisms of $\CC^n$ we can show that $M$ is elliptic. 

\begin{theorem}
\label{thm_affinequotients}
Let $\Gamma \subset \Aut(\CC^n)$ be a discrete group of affine automorphisms of $\CC^n$ acting freely and properly discontinuously on $\CC^n$. Then the quotient manifold $M = \CC^n/\Gamma$ is elliptic.
\end{theorem}


As it is not clear in general whether the quotient $M$ in Theorem \ref{thm_affinequotients} is Stein, we give a direct proof of ellipticity. In the interest of obtaining a more general result which may be of future relevance we first develop a sufficient condition for an arbitrary quotient manifold of the form $\CC^n/\Gamma$ to be elliptic. To this end, let $\Gamma \subset \Aut(\CC^n)$ be any discrete group of automorphisms of $\CC^n$ acting freely and properly discontinuously on $\CC^n$. The quotient $M = \CC^n / \Gamma$ is then a complex $n$-manifold with holomorphic covering map $\pi : \CC^n \to M$. We wish to construct a holomorphic vector bundle of rank $n$ over $M$ as a quotient $(\CC^n \times \CC^n)/\Gamma$ of the trivial vector bundle $\CC^n \times \CC^n$ over the universal cover $\CC^n$. To do so we will extend the action of $\Gamma$ to $\CC^n \times \CC^n$ by identifying copies of the vector bundle fibre $\CC^n$ over points in the same fibre of $\pi$ in such a way that we can produce from $\pi$ a well-defined dominating spray $(\CC^n \times \CC^n) / \Gamma \to M$.

Suppose we are given a holomorphic map $\sigma : \CC^n \times \CC^n \to \CC^n$ with the property that for each $z \in \CC^n$ the map $\sigma_z = \sigma(z, \cdot) : \CC^n \to \CC^n$ is an automorphism of $\CC^n$ satisfying $\sigma_z(0) = z$. Then for each $z \in \CC^n$ the composition $\pi \circ \sigma_z : \CC^n \to M$ satisfies $\pi \circ \sigma_z(0) = \pi(z)$ and is a submersion at $0 \in \CC^n$. We wish to construct an appropriate vector bundle $E \to M$ as a quotient of the trivial bundle $\CC^n \times \CC^n$ on which $\pi \circ \sigma$ is well-defined and then gives a dominating spray onto $M$. To achieve this, suppose $z, z' \in \CC^n$ are such that $\pi(z) = \pi(z')$, so that $z' = \gamma(z)$ for some $\gamma \in \Gamma$. The fibres over $z$ and $z'$ must then be identified so that the following diagram commutes:
\[\xymatrixrowsep{3pc}
\xymatrix{
\{z\}\times\CC^n \ar@{->}[rd]_{\pi \circ \sigma_z} \ar@{->}[rr]^{\lambda_\gamma}&&\{z'\}\times\CC^n  \ar@{->}[ld]^{\pi \circ \sigma_{z'}}\\
&M}
\]
where $\lambda : \Gamma \to \GL(n,\CC)$ is a homomorphism. For this diagram to commute we require $\pi \circ \sigma_z = \pi \circ \sigma_{z'} \circ \lambda_\gamma$, which is equivalent to
\[
	\sigma_{z'} \circ \lambda_\gamma \circ \sigma_z^{-1} \in \Gamma\,.
\]
However, the composition $\sigma_{z'} \circ \lambda_\gamma \circ \sigma_z^{-1}$ maps $z$ to $z'$, so by the freeness of the action of $\Gamma$ we must have
\[
	\sigma_{z'} \circ \lambda_\gamma \circ \sigma_z^{-1} = \gamma\,.
\]

From this discussion we see that, given $\sigma : \CC^n \times \CC^n \to \CC^n$ and $\lambda : \Gamma \to \GL(n, \CC)$ as above, we may extend the action of $\Gamma$ to $\CC^n \times \CC^n$ by the formula
\[
	\gamma \cdot (z,w) = (\gamma(z), \lambda_\gamma w)\,,
\]
where $\gamma \in \Gamma$ and $(z,w) \in \CC^n \times \CC^n$. The quotient $E = (\CC^n \times \CC^n) / \Gamma$ is then a holomorphic vector bundle over $M$. Note that $E$ is flat because $\lambda$ depends only on $\gamma \in \Gamma$, so that the transition functions for $E$ are locally constant. By construction, the map $\pi \circ \sigma$ descends to the quotient $E$ and gives a dominating spray $E \to M$. Thus $M$ is elliptic. We summarise this discussion by the following result.
\begin{proposition}
\label{prop_ellipticquotients}
Let $\Gamma \subset \Aut(\CC^n)$ be a discrete group of automorphisms acting freely and properly discontinuously on $\CC^n$. Suppose $\sigma : \CC^n \times \CC^n \to \CC^n$ is a holomorphic map such that for each $z \in \CC^n$ we have $\sigma_z = \sigma(z,\cdot) \in \Aut(\CC^n)$ and $\sigma_z(0) = z$, and $\lambda : \Gamma \to \GL(n,\CC)$ is a homomorphism. If for all $z \in \CC^n$ and all $\gamma \in \Gamma$ we have $\sigma_{\gamma(z)} \circ \lambda_\gamma \circ \sigma_z^{-1} = \gamma$, then $\CC^n / \Gamma$ is an elliptic manifold.
\end{proposition}

We may now prove Theorem \ref{thm_affinequotients}.

\begin{proof}[Proof of Theorem \ref{thm_affinequotients}]
Let $\Gamma \subset \Aff(\CC^n)$ be a discrete group of affine automorphisms of $\CC^n$. For $\gamma \in \Gamma$ we define $\lambda_\gamma = A$, where $\gamma(z) = Az + b$, $A \in \GL(n,\CC)$ and $b\in \CC^n$. Let $\sigma_z(w) = z + w$. We see that $\lambda$ and $\sigma$ satisfy the conditions of Proposition \ref{prop_ellipticquotients} by the following elementary calculation
\begin{eqnarray*}
	\sigma_{\gamma(z)} \circ \lambda_\gamma \circ \sigma_z^{-1}(w) &=& \sigma_{\gamma(z)} \circ \lambda_\gamma(-z + w) = \sigma_{\gamma(z)}(-Az + Aw) \\&=& -Az + Aw + Az + b = \gamma(w)\,.
\end{eqnarray*}
Therefore $M = \CC^n / \Gamma$ is elliptic.
\end{proof}

Note that in the proof of Theorem \ref{thm_affinequotients}, the homomorphism $\lambda : \Gamma \to \GL(n,\CC)$ is actually the derivative of the affine coordinate change maps for $M$, showing that $E$ is isomorphic to $TM$, the holomorphic tangent bundle of $M$.

\section{New examples of elliptic manifolds}
\label{sec_newexamples}

\noindent
In this section we apply the results of Section \ref{sec_ellipticquotients} in conjunction with an example of Margulis to construct a new example of an elliptic manifold $M$ with $\pi_1(M) \cong F_2$, the free group of rank two, with all higher homotopy groups trivial. By taking products and intermediate covering spaces we obtain a range of other elliptic manifolds with various fundamental groups and vanishing higher homotopy. In particular, given any $k \in \NN \cup \{\aleph_0\}$ there exists an elliptic manifold that is an Eilenberg-MacLane space of type $K(F_k,1)$.

In \cite{Margulis:1983,Margulis:1984} Margulis gave an example of a group $\Gamma \cong F_2$ acting freely and properly discontinuously on $\RR^3$ by affine transformations, providing a counterexample to the conjecture of Milnor that the fundamental group of a complete flat affine manifold is virtually polycyclic \cite{Milnor:1977}. In the example the group $\Gamma$ consists of affine maps each of whose linear part is a hyperbolic element of the group $\SO^0(2,1)$, the connected component of the identity in $\SO(2,1)$. As a result, the quotient manifold $L = \RR^3/\Gamma$ is a complete flat affine Lorentzian manifold with $\pi_1(L) \cong F_2$ and $\pi_n(L) = 0$ for $n > 1$. By the Auslander conjecture in three dimensions \cite{Fried:1983}, $L$ is non-compact.

We now consider the same group $\Gamma$ of affine transformations as automorphisms of $\CC^3$ in the obvious way, namely
\[
\gamma(z) = Az + b\,,\quad z \in \CC^3\,,
\]
where the action of $\gamma \in \Gamma$ on $\RR^3$ is given by $\gamma(x) = Ax + b$ for $x\in \RR^3, A \in \SO^0(2,1)$ and $b \in \RR^3$. It is easy to check that the action of $\Gamma$ on $\CC^3$ remains free and properly discontinuous. We thus obtain a non-compact complex quotient manifold $M = \CC^3/\Gamma$ with $\pi_1(M) \cong F_2$ and $\pi_n(M) = 0$ for $n > 1$. By Theorem \ref{thm_affinequotients}, $M$ is elliptic.

We mention that given a smooth affine manifold $Y$, there is a natural way to define a complex structure on the tangent bundle $TY$ so that $Y$, embedded as the zero-section, is a maximal totally real submanifold of $TY$ \cite{Shimizu:1985}. In this situation $TY$ is commonly called a \emph{complexification} of $Y$. Applying the construction in \cite{Shimizu:1985} to $L$ gives a complex manifold $TL$ that is naturally biholomorphic to $M$, and the biholomorphism restricts to the identity on $L$ under the natural inclusions of $L$ into $TL$ and $M$ respectively. By a result of Grauert \cite{Grauert:1958a} there exists a neighbourhood of $L$ within $M$ that is a Stein manifold, but it is not clear to me whether in fact $M$ itself is Stein, or if there is some other factor which limits the size of Stein neighbourhoods of $L$ within $M$.

By applying some basic properties of ellipticity we obtain a variety of new elliptic manifolds.

\begin{theorem}
\label{prop_freeexamples}
For all $k \in \NN \cup \{\aleph_0\}$ there exists a 3-dimensional elliptic manifold $S$ with $\pi_1(S) \cong F_k$ and $\pi_n(S) = 0$ for all $n>1$.
\end{theorem}
\begin{proof}
Given $k$ as above, let $\Gamma' \subset \Gamma\cong F_2$ be a subgroup isomorphic to $F_k$. Let $S = \CC^3 / \Gamma'$, then $S$ is a covering space of $M$. Using the fact that ellipticity passes up through covering maps \cite[Sec.\ 5]{Forstneric:2010} immediately yields the theorem.
\end{proof}

If we also use the fact that products of elliptic manifolds are elliptic, we obtain a larger collection of new elliptic manifolds as follows. Let $\Cont$ be the smallest collection of groups that contains $F_2$ and that is closed under the operations of taking subgroups and finite direct products of its members. Then for every $G \in \Cont$ there exists an elliptic manifold $S$ with $\pi_1(S) \cong G$ and vanishing higher homotopy groups. A result of Baumslag and Roseblade \cite{Baumslag:1984} on subgroups of $F_2 \times F_2$ then gives the following result.

\begin{corollary}
There exist continuum-many 6-dimensional elliptic manifolds of distinct homotopy type, in fact with mutually non-isomorphic fundamental groups and vanishing higher homotopy.
\end{corollary}

We mention however, the result in the same paper that any \emph{finitely presented} subgroup of the direct product of two free groups is a finite extension of a direct product of two free groups of finite rank.


We note that if it can be shown that $M$ is Stein then it would follow that all of the other elliptic manifolds discussed in this section are also Stein, since products of Stein manifolds are Stein and the property of being Stein passes up via covering maps.

\section{Acyclic embeddings of open Riemann surfaces}
\label{sec_acyclicembeddings}

\noindent
In this section we address a question which arises naturally in the holomorphic homotopy theory of L\'arusson \cite{Larusson:2003,Larusson:2004,Larusson:2005}, namely whether every Stein manifold can be acyclically embedded into an elliptic Stein manifold. We will take all embeddings to be both proper and holomorphic. This question appears very difficult to answer in general, so we restrict ourselves to considering acyclic embeddings of 1-dimensional Stein manifolds, namely open Riemann surfaces.

In \cite{Ritter:2011} it was shown that all open Riemann surfaces with abelian fundamental group acyclically embed into a 2-dimensional elliptic Stein manifold (either $\CC^2$ or $\CC\times\CC^*$ depending on the homotopy type of the Riemann surface). In this section we extend this result to show that every open Riemann surface embeds acyclically into an elliptic manifold. Unfortunately we have not been able to determine so far whether the elliptic targets are Stein. We begin with the following lemma.

\begin{lemma}
\label{lem_homotopicembedding}
Let $f : X \to S \times Z$ be a continuous map where $X$ is a Stein manifold, $S$ is an elliptic manifold and $Z$ is a contractible complex manifold. If $X$ has an embedding $\phi : X \to Z$ then $f$ is homotopic to an embedding $\tilde{f} : X \to S \times Z$.
\end{lemma}
\begin{proof}
Let $\pi_S : S \times Z \to S$ and $\pi_Z : S \times Z \to Z$ be projections onto the first and second components of $S \times Z$ respectively. By Gromov's Oka principle \cite{Gromov:1989}, the continuous map $\pi_S \circ f : X \to S$ is homotopic to a holomorphic map $\psi : X \to S$. Since $Z$ is contractible, the maps $\pi_Z \circ f : X \to Z$ and $\phi : X \to Z$ are homotopic. Consequently, $f = (\pi_S \circ f, \pi_Z \circ f)$ is homotopic to the map $\tilde{f} = (\psi, \phi)$, which is an embedding since $\phi : X \to Z$ is an embedding.
\end{proof}

Using this lemma and the elliptic manifolds from Theorem \ref{prop_freeexamples} we may prove our main result.

\begin{theorem}
Let $X$ be an open Riemann surface. Then $X$ can be acyclically embedded into an elliptic manifold.
\end{theorem}
\begin{proof}
It is well known that the fundamental group of an open Riemann surface $X$ is isomorphic to a free group of rank $k$ for some $k \in \NN \cup \{\aleph_0\}$. Using Theorem \ref{prop_freeexamples} we let $S$ be an elliptic manifold with $\pi_1(S) \cong F_k$, and vanishing higher homotopy groups. Since $X$ and $S$ are both Eilenberg-MacLane spaces of type $K(F_k,1)$, there exists a continuous map $g : X \to S$ which induces the identity homomorphism between the fundamental groups of the two spaces and is thus acyclic \cite[Thm.\ 7.26]{Davis:2001}.

As $X$ is a 1-dimensional Stein manifold there exists an embedding $X \to \CC^3$. By Lemma \ref{lem_homotopicembedding}, the map $f = (g, 0) : X \to S \times \CC^3$ is then homotopic to an embedding $\tilde{f} : X \to S \times \CC^3$. Since $g$ is acyclic, so too is $\tilde{f}$ and the theorem is proved.
\end{proof}

In the above proof, if it could be shown that $S$ is Stein then we would have the stronger result that every 1-dimensional Stein manifold can be acyclically embedded into an elliptic Stein manifold. As mentioned earlier, this was proved in \cite{Ritter:2011} for Riemann surfaces with abelian fundamental group. However, in that paper the target space was 2-dimensional, while in the current paper our targets have dimension 6. It is interesting to ask whether the dimension of our target could be reduced, but it is not clear how this might be achieved. In the current situation we could at best hope to reduce the target dimension to 3, if it were possible to acyclically embed directly into the elliptic manifold $S$.


\end{document}